\documentclass[12pt,a4paper]{article}
\usepackage{amsmath}
\usepackage{amssymb}
\usepackage{amsthm}
\usepackage{cite}
\usepackage{hyperref}

\makeatletter
  \def\@seccntformat#1{%
    \@nameuse{@seccnt@prefix@#1}%
    \@nameuse{the#1}%
    \@nameuse{@seccnt@postfix@#1}%
    \@nameuse{@seccnt@afterskip@#1}}
  \def\@seccnt@prefix@section{}
  \def\@seccnt@postfix@section{.}
  \def\@seccnt@afterskip@section{\hspace{.5em}}
  \def\@seccnt@prefix@subsection{}
  \def\@seccnt@postfix@subsection{.}
  \def\@seccnt@afterskip@subsection{\hspace{.5em}}
\makeatother

\makeatletter
\renewcommand\section{
  \@startsection{section}{3}{\z@}%
  {-3.25ex\@plus -1ex \@minus -.2ex}%
  {1.5ex \@plus .2ex}%
  {\normalfont\normalsize\bfseries\mathversion{bold}}}
\renewcommand\subsection{
  \@startsection{subsection}{3}{\z@}%
  {-3.25ex\@plus -1ex \@minus -.2ex}%
  {1.5ex \@plus .2ex}%
  {\normalfont\normalsize\bfseries\mathversion{bold}}}
\makeatother

\makeatletter \@addtoreset{equation}{section} \makeatother
\renewcommand{\theequation}{\arabic{section}.\arabic{equation}}

\allowdisplaybreaks[1]

\let\oldthebibliography\thebibliography
\renewcommand\thebibliography[1]{
  \oldthebibliography{#1}\setlength{\itemsep}{0.4ex}}

\theoremstyle{plain}
\newtheorem{thm}{Theorem}[section]

\newtheorem{prop}[thm]{Proposition}
\newtheorem{cor}[thm]{Corollary}
\newtheorem{conj}[thm]{Conjecture}

\theoremstyle{definition}
\newtheorem{algo}[thm]{Algorithm}

\newtheorem{dfn}[thm]{Definition}

\theoremstyle{remark}

\theoremstyle{plain}
\newtheorem*{thm*}{Theorem}
\newtheorem*{lem*}{Lemma}
\newtheorem*{prop*}{Proposition}
\newtheorem*{cor*}{Corollary}
\newtheorem*{conj*}{Conjecture}

\theoremstyle{definition}
\newtheorem*{algo*}{Algorithm}
\newtheorem*{ass*}{Assumption}
\newtheorem*{dfn*}{Definition}

\theoremstyle{remark}
\newtheorem*{rem*}{Remark}

\newcommand{\atmark}{{{\scriptsize\hspace{-.1em}%
  \raisebox{-.65ex}{\Large$\circ$}\hspace{-.83em}%
  \raisebox{-.03ex}{\footnotesize$a$}}\hspace{.15em}}}

\newcommand{\nn}{\nonumber}

\newcommand{\grp}[1]{\mathrm{#1}}
\newcommand{\bvec}[1]{\boldsymbol{#1}}

\newcommand{\bbC}{\mathbb{C}}
\newcommand{\bbZ}{\mathbb{Z}}
\newcommand{\bbH}{\mathbb{H}}
\newcommand{\ri}{{i}}
\newcommand{\varth}{\vartheta}

\newcommand{\rw}{\mathrm{w}}
\newcommand{\cR}{\mathcal{R}}

\newcommand{\vecal}{{\boldsymbol{\alpha}}}
\newcommand{\vece}{{\boldsymbol{\mathrm{e}}}}
\newcommand{\vecL}{{\boldsymbol{\Lambda}}}
\newcommand{\vecv}{{\boldsymbol{v}}}
\newcommand{\vecw}{{\boldsymbol{w}}}
\newcommand{\vecz}{{\boldsymbol{z}}}
\newcommand{\veczero}{{\boldsymbol{0}}}

\newcommand{\lb}{\mathrm{b}}
\newcommand{\Jlb}{J^{E_8,\lb}_{*}}
\newcommand{\dlb}{d^\lb}
\newcommand{\bx}[1]{\makebox[1.1em][c]{#1}}

\begin{document}


\def\papertitlepage{\baselineskip 3.5ex \thispagestyle{empty}}
\def\preprinumber#1#2{\hfill
\begin{minipage}{1.22in}
#1 \par\noindent #2
\end{minipage}}

%
\papertitlepage
\setcounter{page}{0}
\preprinumber{arXiv:2201.06895}{}
\vskip 2ex
\vfill
\begin{center}
{\large\bf\mathversion{bold}
Algebraic construction of Weyl invariant $E_8$ Jacobi forms
}
\end{center}
\vfill
\baselineskip=3.5ex
\begin{center}
Kazuhiro Sakai\\

{\small
\vskip 6ex
{\it Institute of Physics, Meiji Gakuin University,
Yokohama 244-8539, Japan}\\
\vskip 1ex
{\tt kzhrsakai\atmark gmail.com}

}
\end{center}
\vfill
\baselineskip=3.5ex
\begin{center} {\bf Abstract} \end{center}

We study the ring of Weyl invariant $E_8$ weak Jacobi forms.
Wang recently proved that the ring is not a polynomial algebra.
We consider a polynomial algebra which contains the ring as a subset
and clarify the necessary and sufficient condition for
an element of the polynomial algebra
to be a Weyl invariant $E_8$ weak Jacobi form.
This serves as a new algorithm for constructing all the Jacobi forms
of given weight and index. The algorithm is purely algebraic
and does not require Fourier expansion.
Using this algorithm we determine the generators of the free module of
Weyl invariant $E_8$ weak Jacobi forms of given index $m$
for $m\le 20$. We also determine the lowest weight generators
of the free module of index $m$ for $m\le 28$.
Our results support the lower bound conjecture of Sun and Wang
and prove explicitly that there exist generators of
the ring of Weyl invariant $E_8$ weak Jacobi forms
of weight $-4m$ and index $m$ for all $12\le m \le 28$.

\vfill
\noindent
January 2022


\setcounter{page}{0}
\newpage
\renewcommand{\thefootnote}{\arabic{footnote}}
\setcounter{footnote}{0}
\setcounter{section}{0}
\baselineskip = 3.5ex
\pagestyle{plain}

\section{Introduction and summary}\label{sec:intro}

Jacobi forms are functions which have the characteristics of
both an elliptic function and a modular form.
Eichler and Zagier initiated the systematic study of Jacobi forms
\cite{EichlerZagier}.
Wirthm\"uller investigated a generalization
of Jacobi forms associated with root systems \cite{Wirthmuller}.
Let $R$ be an irreducible root system of rank $r$
and $W(R)$ the Weyl group of $R$.
A Weyl invariant $R$ Jacobi form, or
a $W(R)$-invariant Jacobi form is a holomorphic function
of variables $\tau\in\bbH$ and $\vecz\in \bbC^r$
which is quasi-periodic in $\vecz$,
invariant under $W(R)$ acting on $\vecz$
and satisfies the modular transformation law.

$W(R)$-invariant Jacobi forms of integral weight and integral index
form a bigraded algebra $J^R_{*,*}$ over the ring $M_*$ of
$\grp{SL}_2(\bbZ)$ modular forms.
Wirthm\"uller proved that $J^R_{*,*}$
is a polynomial algebra for any irreducible root system $R$
except $E_8$ \cite{Wirthmuller}.
Generators of the polynomial ring have been explicitly constructed
for all irreducible root systems $R$ other than $E_8$
\cite{Wirthmuller,Satake:1993cp,BertolaThesis,Bertola1,Sakai:2017ihc,
Adler:2019ysr,Adler:2021}.
On the other hand, for $R=E_8$ the structure of the bigraded algebra
$J^{E_8}_{*,*}$ has not yet been fully understood.

The investigation of $W(E_8)$-invariant Jacobi forms was first
developed in physics. The systematic construction of 
$W(E_8)$-invariant Jacobi forms was initiated
in \cite{Minahan:1998vr} in the study of E-strings.
In \cite{Eguchi:2002fc},
nine $W(E_8)$-invariant meromorphic Jacobi forms
$a_2,a_3,a_4,b_1,b_2,b_3,b_4,b_5,b_6$ were constructed
as the coefficients of the Seiberg--Witten curve
for the E-string theory. Among many applications,
the curve can be used in particular
to generate the topological string partition function
of the Calabi--Yau threefolds known as
the local $\frac{1}{2}$K3 or the local rational elliptic surfaces
\cite{Sakai:2011xg}.
$a_i,b_j$ are meromorphic because they have poles
at the zero points of Eisenstein series $E_4(\tau)$.
In \cite{Sakai:2011xg}, $a_i,b_j$ were expressed in terms of
nine simple $W(E_8)$-invariant holomorphic Jacobi forms
$A_1,A_2,A_3,A_4,A_5,B_2,B_3,B_4,B_6$
and $E_4,E_6$.

In physics, it is deduced by symmetry consideration
\cite{Minahan:1998vr} that
the elliptic genus $Z_n$ of $n$ E-strings
multiplied by $\Delta^{n/2}$ ($\Delta=(E_4^3-E_6^2)/1728$)
has to be a $W(E_8)$-invariant quasi Jacobi form
(i.e.~involving $E_2(\tau)$) of weight $6n-2$ and index $n$.
Moreover, a prescription to express $\Delta^{n/2}Z_n$
as a polynomial of $a_i,b_j,E_2,E_4,E_6$ for any $n$
was given in \cite{Eguchi:2002fc}.
Therefore, a natural guess from physics is that
$a_i,b_j$
--- though they are meromorphic and
are not exactly Jacobi forms themselves ---
play the role of ``generators'' of
$W(E_8)$-invariant Jacobi forms,
in the sense that any $W(E_8)$-invariant Jacobi form
is written as a polynomial of $a_i,b_j,E_4,E_6$.
This was explicitly conjectured in \cite{Sakai:2017ihc}.
A proof of this conjecture (Theorem~\ref{thm:subset})
is one of the main results of this paper.

In \cite{Wang:2018fil}, 
Wang initiated the detailed investigation
of the bigraded algebra $J^{E_8}_{*,*}$
of $W(E_8)$-invariant Jacobi forms.
In particular, it was explicitly proved that $J^{E_8}_{*,*}$
is not a polynomial algebra over $M_*$.
In \cite{Wang:2018fil} and other earlier works
such as \cite{Huang:2013yta,DelZotto:2017mee},
$W(E_8)$-invariant Jacobi forms are written
in terms of the above mentioned
$A_i,B_j$ introduced in \cite{Sakai:2011xg}.
It is natural to use $A_i,B_j$ rather than $a_i,b_j$
as the building blocks because 
$A_i,B_j$ are simpler and they themselves are Jacobi forms.
On the other hand,
it was pointed out in \cite{Huang:2013yta} that 
certain $W(E_8)$-invariant Jacobi forms
cannot be expressed as polynomials of $A_i,B_j$
unless further multiplied by $E_4$.
In \cite{DelZotto:2017mee} it was found that
the $W(E_8)$-invariant holomorphic Jacobi form
of weight $16$ and index $5$ given by
\begin{align}
P_{16,5}
 =864A_1^3A_2+3825A_1B_2^2-770E_6A_3B_2-840E_6A_2B_3+60E_6A_1B_4
  +21E_6^2A_5
\label{eq:P16c5}
\end{align}
vanishes at the zero points of $E_4$.
The authors of \cite{DelZotto:2017mee} then conjectured that
the quotient $P_{16,5}/E_4$ is also a $W(E_8)$-invariant
holomorphic Jacobi form.
They further conjectured that any polynomial of $E_6,A_i,B_j$
which vanishes at the zero points of $E_4$
is divisible by $P_{16,5}$.
These conjectures were proved recently
by Sun and Wang \cite{Sun:2021ije}.
Moreover, Sun and Wang proved that
every $W(E_8)$-invariant Jacobi form of index $t$ can be expressed
uniquely as
\begin{align}
\frac{\sum_{j=0}^{t_1}\left({P_{16,5}}/{E_4}\right)^{t_1-j}P_j}
     {\Delta^{N_t}},
\label{eq:SunWangform0}
\end{align}
where $t_1,N_t$ are certain non-negative integers determined by $t$,
$\{P_j\}_{j=0}^{t_1-1}$ are polynomials of $E_6,A_i,B_j$,
and $P_{t_1}$ is a polynomial of $E_4,E_6,A_i,B_j$.

This theorem is a powerful tool to investigate
the structure of the bigraded algebra $J^{E_8}_{*,*}$
of $W(E_8)$-invariant Jacobi forms.
In fact, it is of practical use in constructing
$W(E_8)$-invariant Jacobi forms, in particular, of higher index:
For given index, the polynomial ansatz has a finite number of
coefficients, which are to be constrained so that the Fourier
expansion of the ansatz satisfies the condition of
being a Jacobi form.
The condition boils down to solving linear equations.
Sun and Wang determined the generators of the free module
of $W(E_8)$-invariant Jacobi forms of given index $m$ for $m\le 13$
\cite{Sun:2021ije}.
A technical drawback of this construction is that
the Fourier expansion
involves manipulation of Weyl orbits of $E_8$,
which gets more and more complicated
when one constructs Jacobi forms of higher index.

In this paper we propose another efficient method of constructing
$W(E_8)$-invariant Jacobi forms.
Our method is purely algebraic
and does not require Fourier expansion.
We consider the polynomial algebra $\cR$
generated by $a_i,b_j$ over $M_*=\bbC[E_4,E_6]$.
By our Theorem~\ref{thm:subset},
$\cR$ contains the bigraded algebra $J^{E_8}_{*,*}$
as a subset. In fact, any element of $\cR$ satisfies
almost all of the properties of $W(E_8)$-invariant Jacobi forms.
The only issue is the holomorphicity: some of the elements have poles
at the zero points of $E_4$,
which prevents them from being Jacobi forms.
Therefore, we have only to check whether
a given element of $\cR$ is holomorphic at these points or not.
This can be done easily by using our Theorem~\ref{thm:E8cond}:
It turns out that the necessary and sufficient condition for
an element of $\cR$ to be a $W(E_8)$-invariant Jacobi form
is that it can be written as a polynomial of
$\Delta^{-1},E_4,E_6,A_i,B_j,(P_{16,5}/E_4)$.
This gives an efficient algorithm for
constructing all the Jacobi forms of given weight and index.

Using this algorithm
we will determine the generators of the free module $J^{E_8}_{*,m}$
of $W(E_8)$-invariant weak Jacobi forms of index $m$ for $m\le 20$.
We will also determine the lowest weight generators
of $J^{E_8}_{*,m}$ for $m\le 28$. Our results confirm
the lower bound conjecture of Sun and Wang \cite{Sun:2021ije}
that the weight of non-zero
$W(E_8)$-invariant weak Jacobi forms of index $m$ is not
less than $-4m$.
Moreover, our results prove explicitly that 
the generators of the ring $J^{E_8}_{*,*}$ must include those
of weight $-4m$ and index $m$ for all $12\le m \le 28$.
This means that the structure of $J^{E_8}_{*,*}$
is highly complicated,
as already recognized in \cite{Sun:2021ije},
but perhaps more highly than expected.

The paper is organized as follows.
In section~\ref{sec:jacobi} we collect known results
about $W(E_8)$-invariant Jacobi forms.
In section~\ref{sec:proof} we prove our main theorems.
In section~\ref{sec:results} we explain our algorithm
in detail and present the results
about the generators of $J^{E_8}_{*,m}$ and $J^{E_8}_{*,*}$.
There are two appendices on conventions.

\section{Preliminaries}\label{sec:jacobi}

\subsection{Definitions}

Let $\bbH=\{\tau\in\bbC|\mathrm{Im}\,\tau>0\}$ be the upper half plane
and set $q=e^{2\pi\ri\tau}$.
Let $E_4(\tau),E_6(\tau)$ denote the Eisenstein series
of weight $4,6$ respectively (see Appendix~\ref{app:functions}).
It is well known that $E_4,E_6$ generate
the graded algebra $M_*$ of all $\grp{SL}_2(\bbZ)$ modular forms:
\begin{align}
M_*=\bbC[E_4,E_6].
\end{align}
We will also use the cusp form
\begin{align}
\Delta=\eta^{24}=\frac{1}{1728}\left(E_4^3-E_6^2\right)
\label{eq:Delta}
\end{align}
frequently, where $\eta(\tau)$ is the Dedekind eta function.

\begin{dfn}[Weyl invariant Jacobi form]\label{def:Jacobi}
Let $R$ be an irreducible root system of rank $r$
and $W(R)$ the Weyl group of $R$.
Let $L_R$ be the root lattice of $R$
and $L_R^*$ the dual lattice of $L_R$.
When $L_R$ is odd, we rescale its bilinear form by 2
so that it becomes an even lattice.
A holomorphic function $\varphi_{k,m}:\bbH\times\bbC^r\to \bbC$
is called a $W(R)$-invariant weak Jacobi form
of weight $k$ and index $m$ ($k\in\bbZ,\ m\in\bbZ_{\ge 0}$) 
if it satisfies
the following properties \cite{EichlerZagier, Wirthmuller}:
\renewcommand{\theenumi}{\roman{enumi}}
\renewcommand{\labelenumi}{(\theenumi)}
\begin{enumerate}
\item Weyl invariance:
\begin{align}
\varphi_{k,m}(\tau,w(\vecz)) = \varphi_{k,m}(\tau,\vecz),\qquad
w\in W(R).
\label{Weylinv}
\end{align}

\item Quasi-periodicity:
\begin{align}
\varphi_{k,m}(\tau,\vecz+\tau\bvec{\alpha}+\bvec{\beta})
=e^{-m \pi \ri (\tau\bvec{\alpha}^2+2\vecz\cdot\bvec{\alpha})}
\varphi_{k,m}(\tau,\vecz),\qquad
\bvec{\alpha},\bvec{\beta}\in L_R.
\end{align}

\item Modular transformation law:
\begin{align}\label{Modularprop}
&
\varphi_{k,m}\left(
\frac{a\tau+b}{c\tau+d}\,,\frac{\vecz}{c\tau+d}\right)
=(c\tau+d)^k\exp\left(m\pi \ri\frac{c}{c\tau+d}\,\vecz^2\right)
\varphi_{k,m}(\tau,\vecz),\\[1ex]
&
\Bigl(\begin{array}{cc}a&b\\ c&d\end{array}\Bigr)
\in \grp{SL}_2(\bbZ).\nn
\end{align}

\item $\varphi_{k,m}(\tau,\vecz)$ admits
a Fourier expansion of the form
\begin{align}\label{Fourierform}
\varphi_{k,m}(\tau,\vecz)
=\sum_{n=0}^\infty
 \sum_{\vecw\in L_R^*}
 c(n,\vecw)e^{2\pi \ri\vecw\cdot\vecz}q^n.
\end{align}
\end{enumerate}
If $\varphi_{k,m}(\tau,\vecz)$ further satisfies the condition
that the coefficients $c(n,\vecw)$
of the Fourier expansion (\ref{Fourierform})
vanish unless $\vecw^2\le 2mn$,
it is called a $W(R)$-invariant holomorphic Jacobi form.
In this paper a Jacobi form means a weak Jacobi form
unless otherwise specified.
\end{dfn}
\begin{rem*}
In this paper we also introduce the notion of
meromorphic Jacobi forms:
we call function $\psi(\tau,\vecz)$ a meromorphic Jacobi form
if $\psi$ itself is not a Jacobi form but there exists
a modular form $f\in M_*$ such that
$f\psi$ is a Jacobi form.
\end{rem*}

Let the vector space of $W(R)$-invariant weak Jacobi forms
of weight $k$ and index $m$ be denoted by
\begin{align}
J^{R}_{k,m}.
\end{align}
$J^{R}_{k,m}$ constitute the free module
\begin{align}
J^{R}_{*,m}
 :=\bigoplus_{k\in\bbZ}J^{R}_{k,m}
\end{align}
and the bigraded algebra
\begin{align}
J^{R}_{*,*}
 :=\bigoplus_{m=0}^{\infty}J^{R}_{*,m}
\end{align}
over $M_*$.
In this paper we will study
$J^{R}_{*,m}$ and $J^{R}_{*,*}$ for $R=E_8$.

\subsection{$W(E_8)$-invariant Jacobi forms}

Nine independent $W(E_8)$-invariant holomorphic Jacobi forms
were constructed in \cite{Sakai:2011xg}.
The basic building block is
the theta function of the root lattice $L_{E_8}$:
\begin{align}
\begin{aligned}
\Theta_{E_8}(\tau,\vecz)
 :=\sum_{\vecw\in L_{E_8}}
   \exp\left(\pi \ri\tau\vecw^2
   +2\pi \ri\vecz\cdot\vecw\right)
 =\frac{1}{2}\sum_{k=1}^4\prod_{j=1}^8\varth_k(z_j,\tau).
\end{aligned}
\end{align}
$\varth_k(z,\tau)$ are the Jacobi theta functions
(see Appendix~\ref{app:functions}).
The nine $W(E_8)$-invariant holomorphic Jacobi forms
are given by \cite[Appendix A]{Sakai:2011xg}
\begin{align}
A_1(\tau,\vecz)&=\Theta_{E_8}(\tau,\vecz),\qquad
A_4(\tau,\vecz)=A_1(\tau,2\vecz),\nn\\
A_m(\tau,\vecz)&=\tfrac{m^3}{m^3+1}\left(
  A_1(m\tau,m\vecz)
  +\tfrac{1}{m^4}\mbox{$\sum_{k=0}^{m-1}$}
  A_1(\tfrac{\tau+k}{m},\vecz)
 \right),\qquad m=2,3,5,\nn\\
B_2(\tau,\vecz)&=\tfrac{32}{5}\left(
 e_1(\tau)A_1(2\tau,2\vecz)
 +\tfrac{1}{2^4}e_3(\tau)A_1(\tfrac{\tau}{2},\vecz)
 +\tfrac{1}{2^4}e_2(\tau)A_1(\tfrac{\tau+1}{2},\vecz)\right),\nn\\
B_3(\tau,\vecz)&=\tfrac{81}{80}\left(
 h_0(\tau)^2A_1(3\tau,3\vecz)
  -\tfrac{1}{3^5}\mbox{$\sum_{k=0}^{2}$}h_0(\tfrac{\tau+k}{3})^2
  A_1(\tfrac{\tau+k}{3},\vecz)\right),\nn\\
B_4(\tau,\vecz)&=\tfrac{16}{15}\left(
 \varth_4(2\tau)^4A_1(4\tau,4\vecz)
 -\tfrac{1}{2^4}\varth_4(2\tau)^4
  A_1(\tau+\tfrac{1}{2},2\vecz)\right.\nn\\
&\hspace{2em}
 \left.
 -\tfrac{1}{2^2\cdot 4^4}\mbox{$\sum_{k=0}^{3}$}
  \varth_2(\tfrac{\tau+k}{2})^4
  A_1(\tfrac{\tau+k}{4},\vecz)\right),\nn\\
B_6(\tau,\vecz)&=\tfrac{9}{10}\left(
  h_0(\tau)^2A_1(6\tau,6\vecz)
 +\tfrac{1}{2^4}\mbox{$\sum_{k=0}^{1}$}
  h_0(\tau+k)^2A_1(\tfrac{3\tau+3k}{2},3\vecz)\right.\nn\\
&\hspace{2em}\left.
 -\tfrac{1}{3\cdot 3^4}\mbox{$\sum_{k=0}^{2}$}
  h_0(\tfrac{\tau+k}{3})^2A_1(\tfrac{2\tau+2k}{3},2\vecz)\right.\nn\\
&\hspace{2em}\left.
 -\tfrac{1}{3\cdot 6^4}\mbox{$\sum_{k=0}^{5}$}
  h_0(\tfrac{\tau+k}{3})^2A_1(\tfrac{\tau+k}{6},\vecz)\right).
\label{E8AB}
\end{align}
Here, functions $e_j(\tau)$ and $h_0(\tau)$ are defined as
\begin{align}
\begin{aligned}
e_1(\tau)&:=
 \tfrac{1}{12}\left(\varth_3(\tau)^4+\varth_4(\tau)^4\right),\\
e_2(\tau)&:=
 \tfrac{1}{12}\left(\varth_2(\tau)^4-\varth_4(\tau)^4\right),\\
e_3(\tau)&:=
 \tfrac{1}{12}\left(-\varth_2(\tau)^4-\varth_3(\tau)^4\right),\\
h_0(\tau)&:=
 \varth_3(2\tau)\varth_3(6\tau)+\varth_2(2\tau)\varth_2(6\tau).
\end{aligned}
\end{align}
$A_m,B_m$ are of weight $4,6$ respectively and of index $m$.
If we set $\vecz=\veczero$,
these Jacobi forms reduce to the Eisenstein series $E_4,E_6$:
\begin{align}
A_m(\tau,\veczero)=E_4(\tau),\qquad
B_m(\tau,\veczero)=E_6(\tau).
\end{align}

Based on the conjectures of \cite{DelZotto:2017mee},
Sun and Wang proved the following theorem,
which will be used in the proof of our main theorems.
\begin{thm}[{Sun and Wang \cite[Theorem 1.1]{Sun:2021ije}}]
\label{thm:SunWang}
\renewcommand{\theenumi}{\arabic{enumi}}
\renewcommand{\labelenumi}{$(\theenumi)$}
~\\[-4ex]
\begin{enumerate}
\item $P_{16,5}/E_4$ with $P_{16,5}$ given by \eqref{eq:P16c5}
is a $W(E_8)$-invariant holomorphic Jacobi form of
weight $12$ and index $5$.
\item For any
$W(E_8)$-invariant Jacobi form $P\in\bbC[E_6,\{A_i\},\{B_j\}]$,
if $P/E_4$ is holomorphic on $\bbH\times\bbC^8$, then
\begin{align}
\frac{P}{P_{16,5}}\in
\bbC[E_6,\{A_i\},\{B_j\}].
\end{align}
\item Every $W(E_8)$-invariant Jacobi form
of index $t$ can be expressed uniquely as
\begin{align}
\frac{\sum_{j=0}^{t_1}\left({P_{16,5}}/{E_4}\right)^{t_1-j}P_j}
     {\Delta^{N_t}},
\label{eq:SunWangform}
\end{align}
where $t_1,N_t\in\bbZ_{\ge 0}$ are such that
$t_1=[t/5]$,
$N_t-5t_0=0,0,1,2,3,3$ for $t-6t_0=0,1,2,3,4,5$
respectively with $t_0=[t/6]$,
$[x]$ is the integer part of $x$ and
\begin{align}
\{P_j\}_{j=0}^{t_1-1}&\in \bbC[E_6,\{A_i\},\{B_j\}],\qquad
P_{t_1}\in \bbC[E_4,E_6,\{A_i\},\{B_j\}].
\end{align}
\end{enumerate}
\end{thm}
%

\subsection{$W(E_8)$-invariant meromorphic Jacobi forms}

In \cite{Eguchi:2002fc} nine meromorphic Jacobi forms
$\{a_i\}_{i=2}^4,\,\{b_j\}_{j=1}^6$ were explicitly constructed.
Later they were expressed in terms of the above $A_i,B_j$
as \cite[Appendix A]{Sakai:2011xg}
\begin{align}
a_2&=
  \frac{6}{E_4\Delta}\Bigl(-E_4A_2+A_1^2\Bigr),\nn\\
a_3&=\frac{1}{9E_4^2\Delta^2}
  \Bigl(-7E_4^2E_6A_3-20E_4^3B_3
       -9E_4E_6A_1A_2+30E_4^2A_1B_2+6E_6A_1^3\Bigr),\nn\\
a_4&=\frac{1}{864E_4^3\Delta^3}
  \Bigl((E_4^6-E_4^3E_6^2)A_4+(56E_4^5-56E_4^2E_6^2)A_1A_3
  -27E_4^5A_2^2\nn\\
&\hspace{1em}
  -90E_4^3E_6A_2B_2-75E_4^4B_2^2+(180E_4^4-36E_4E_6^2)A_1^2A_2\nn\\
&\hspace{1em}
  +240E_4^2E_6A_1^2B_2
  +(-210E_4^3+18E_6^2)A_1^4\Bigr),\nn\\
b_1&=-\frac{4}{E_4}A_1,\qquad
b_2=
  \frac{5}{6E_4^2\Delta}\Bigl(E_4^2B_2-E_6A_1^2\Bigr),\nn\\
b_3&=\frac{1}{108E_4^3\Delta^2}
  \Bigl(-7E_4^5A_3-20E_4^3E_6B_3\nn\\
&\hspace{1em}
  -9E_4^4A_1A_2+30E_4^2E_6A_1B_2+(16E_4^3-10E_6^2)A_1^3\Bigr),\nn\\
b_4&=\frac{1}{1728E_4^4\Delta^3}
  \Bigl((-5E_4^7+5E_4^4E_6^2)B_4
  +(80E_4^6-80E_4^3E_6^2)A_1B_3\nn\\
&\hspace{1em}
  +9E_4^5E_6A_2^2+30E_4^6A_2B_2+25E_4^4E_6B_2^2
  -48E_4^4E_6A_1^2A_2\nn\\
&\hspace{1em}
  +(-140E_4^5+60E_4^2E_6^2)A_1^2B_2
  +(74E_4^3E_6-10E_6^3)A_1^4\Bigr),\nn\\
\noalign{\break}
b_5&=\frac{1}{72E_4^5\Delta^3}\Bigl(
  (-21E_4^7+21E_4^4E_6^2)A_5-294E_4^6A_2A_3-770E_4^4E_6B_2A_3\nn\\
&\hspace{1em}
  -840E_4^4E_6A_2B_3-2200E_4^5B_2B_3+168E_4^5A_1^2A_3
  +480E_4^3E_6A_1^2B_3\nn\\
&\hspace{1em}
  -621E_4^5A_1A_2^2+3525E_4^4A_1B_2^2
  +1224E_4^4A_1^3A_2-240E_4^2E_6A_1^3B_2\nn\\
&\hspace{1em}
  +(-456E_4^3+24E_6^2)A_1^5
\Bigr),\nn\\
b_6&=\frac{1}{13436928E_4^6\Delta^5}\Bigl(
  (-20E_4^{12}+40E_4^9E_6^2-20E_4^6E_6^4)B_6\nn\\
&\hspace{1em}
 +(-189E_4^{10}E_6+378E_4^7E_6^3-189E_4^4E_6^5)A_1A_5\nn\\
&\hspace{1em}
 +(-9E_4^{10}E_6+9E_4^7E_6^3)A_2A_4
 +(-15E_4^{11}+15E_4^8E_6^2)B_2A_4\nn\\
&\hspace{1em}
 +(-180E_4^{11}+180E_4^8E_6^2)A_2B_4
 +(-300E_4^9E_6+300E_4^6E_6^3)B_2B_4\nn\\
&\hspace{1em}
 +(22E_4^9E_6-22E_4^6E_6^3)A_1^2A_4
 +(150E_4^{10}+120E_4^7E_6^2-270E_4^4E_6^4)A_1^2B_4\nn\\
&\hspace{1em}
 +(196E_4^{10}E_6-196E_4^7E_6^3)A_3^2
 +(1120E_4^{11}-1120E_4^8E_6^2)A_3B_3\nn\\
&\hspace{1em}
 +(1600E_4^9E_6-1600E_4^6E_6^3)B_3^2
 +(-2982E_4^9E_6+2982E_4^6E_6^3)A_1A_2A_3\nn\\
&\hspace{1em}
 +(-2520E_4^{10}-4410E_4^7E_6^2+6930E_4^4E_6^4)A_1B_2A_3\nn\\
&\hspace{1em}
 +(3360E_4^{10}-10920E_4^7E_6^2+7560E_4^4E_6^4)A_1A_2B_3\nn\\
&\hspace{1em}
 +(-19800E_4^8E_6+19800E_4^5E_6^3)A_1B_2B_3
 +(2016E_4^8E_6-2016E_4^5E_6^3)A_1^3A_3\nn\\
&\hspace{1em}
 +(-5920E_4^9+7360E_4^6E_6^2-1440E_4^3E_6^4)A_1^3B_3
 +(405E_4^9E_6+162E_4^6E_6^3)A_2^3\nn\\
&\hspace{1em}
 +(1215E_4^{10}+1620E_4^7E_6^2)A_2^2B_2
 +4725E_4^8E_6A_2B_2^2\nn\\
&\hspace{1em}
 +(1125E_4^9+1500E_4^6E_6^2)B_2^3
 +(-9477E_4^8E_6+5103E_4^5E_6^3)A_1^2A_2^2\nn\\
&\hspace{1em}
 +(-9180E_4^9-5400E_4^6E_6^2)A_1^2A_2B_2
 +(20925E_4^7E_6-33075E_4^4E_6^3)A_1^2B_2^2\nn\\
&\hspace{1em}
 +(20304E_4^7E_6-9072E_4^4E_6^3)A_1^4A_2\nn\\
&\hspace{1em}
 +(12780E_4^8+5400E_4^5E_6^2+540E_4^2E_6^4)A_1^4B_2\nn\\
&\hspace{1em}
 +(-11076E_4^6E_6+1512E_4^3E_6^3-36E_6^5)A_1^6
\Bigr).
\label{eq:abinABfull}
\end{align}
One can easily invert these relations and express
$A_i,B_j$ in terms of $a_k,b_l$ as
\begin{align}
A_1&=
 -\frac{E_4b_1}{4},\qquad
A_2=
 \frac{3E_4b_1^2-8\Delta a_2}{48},\nn\\
A_3&=
 \frac{-21E_4b_1^3-12\Delta E_4b_3+\Delta E_6a_3-72\Delta a_2b_1}
      {1344},\nn\\
A_4&=
 \frac{1}{2304}(
  \Delta E_4^2a_2^2+9E_4b_1^4-288\Delta E_4b_1b_3+144\Delta
  E_4b_2^2-24\Delta E_6a_2b_2+24\Delta E_6a_3b_1\nn\\
&\hspace{2em}
  +1296\Delta a_2b_1^2+1152\Delta^2a_4),\nn\\
A_5&=
 \frac{1}{64512}(3\Delta E_4^2a_2^2b_1-63E_4b_1^5+216\Delta E_4b_1^2b_3
  -144\Delta E_4b_1b_2^2-24\Delta E_6a_2b_1b_2\nn\\
&\hspace{2em}
  +110\Delta E_6a_3b_1^2-1200\Delta a_2b_1^3-128\Delta^2E_4b_5
  -1344\Delta^2a_2b_3+2112\Delta^2a_3b_2),\nn\\[1ex]
B_2&=
 \frac{5E_6b_1^2+96\Delta b_2}{80},\qquad
B_3=
 \frac{-\Delta E_4^2a_3-60E_6b_1^3+12\Delta E_6b_3-1728\Delta b_1b_2}
      {3840},\nn\\
B_4&=
 \frac{1}{34560}(-24\Delta E_4^2a_2b_2+36\Delta E_4^2a_3b_1
  +\Delta E_4E_6a_2^2+135E_6b_1^4-432\Delta E_6b_1b_3\nn\\
&\hspace{2em}
  +144\Delta E_6b_2^2+5184\Delta b_1^2b_2-6912\Delta^2b_4),\nn\\
B_6&=
 \frac{1}{552960}(-\Delta E_4^2E_6a_4b_1^2+72\Delta E_4^2a_2b_1^2b_2
  -216\Delta E_4^2a_3b_1^3-9\Delta E_4E_6a_2^2b_1^2\nn\\
&\hspace{2em}
  +135E_6b_1^6-96\Delta^2E_4^2a_2b_4+72\Delta^2E_4^2a_3b_3
  -144\Delta^2E_4^2a_4b_2+12\Delta^2E_4E_6a_2a_4\nn\\
&\hspace{2em}
  -3\Delta^2E_4E_6a_3^2-144\Delta^2E_4a_2^2b_2+288\Delta^2E_4a_2a_3b_1
  +12\Delta^2E_6a_2^3+12\Delta E_6^2b_1^2b_4\nn\\
&\hspace{2em}
  -216\Delta E_6b_1^3b_3+7776\Delta b_1^4b_2-2592\Delta^2E_6b_1b_5
  +1152\Delta^2E_6b_2b_4-432\Delta^2E_6b_3^2\nn\\
&\hspace{2em}
  +10368\Delta^2b_1^2b_4-124416\Delta^3b_6).
\label{eq:ABinab}
\end{align}
\begin{prop}
\label{prop:ABinab}
$A_i,B_j$ are polynomials of $E_4,E_6,a_k,b_l$.
\end{prop}
\begin{proof}
This is clear from \eqref{eq:Delta} and \eqref{eq:ABinab}.
\end{proof}

$a_m,b_m$ are of weight $4-6m,6-6m$ respectively and of index $m$.
They are meromorphic and are not exactly Jacobi forms.
Indeed, the lowest weight of non-zero $W(E_8)$-invariant
weak Jacobi forms of index $m$ for $m=1,2,3,4,5,6$
is $4$, $-4$, $-8$, $-16$, $-16$, $-24$ respectively
\cite{Wang:2018fil,Sun:2021ije} (or see \eqref{eq:Pwlists}),
but the weight of every $a_m,b_m$ is lower than this bound. 
More specifically, $a_m,b_m$ have poles
at the zero points of $E_4$.
On the other hand, they behave well at the cusp:
\begin{prop}
\label{prop:abqexp}
$a_i,b_j$ admit a Fourier expansion
of the form
\begin{align}
a_i(\tau,\vecz)=\sum_{n=0}^\infty a_i^{(n)}(\vecz)q^n,\qquad
b_j(\tau,\vecz)=\sum_{n=0}^\infty b_j^{(n)}(\vecz)q^n.
\end{align}
\end{prop}
\begin{proof}
By construction the $q$-expansions of $a_i,b_j$ are integral.
The absence of negative powers can be shown by direct calculation.
\end{proof}
$a_i^{(0)},b_j^{(0)}$ were explicitly computed
in \cite[Appendix~B]{Eguchi:2002fc}:
\begin{align}
a_2^{(0)}
 &=-\frac{2}{3}w_1+12w_8-1440,\qquad
a_3^{(0)}
  =-2w_2+96w_1-1152w_8+103680,\nn\\
a_4^{(0)}
 &=\frac{4}{3}w_1^2-4w_3-16w_6-48w_1w_8-144w_8^2\nn\\
 &\hspace{1em}
  +400w_2+1440w_7+1728w_1+41472w_8-2073600,\nn\\
b_1^{(0)}
 &=-4,\qquad
b_2^{(0)}
  =-\frac{1}{18}w_1-3w_8+840,\nn\\
b_3^{(0)}
 &=-\frac{1}{6}w_2-4w_7-8w_1+528w_8-79680,\nn\\
\noalign{\break}
b_4^{(0)}
 &=\frac{2}{9}w_1^2-\frac{1}{3}w_3-\frac{16}{3}w_6
  -24w_1w_8-120w_8^2\nn\\
 &\hspace{1em}
  +\frac{424}{3}w_2+1272w_7+4608w_1-25920w_8+3939840,\nn\\
b_5^{(0)}
 &=\frac{2}{3}w_1w_2-4w_5-16w_1w_7+64w_2w_8+288w_7w_8
  -96w_1^2-60w_3-160w_6\nn\\
 &\hspace{1em}
  +3456w_8^2+800w_2-24480w_7-108480w_1+933120w_8-97873920,\nn\\
b_6^{(0)}
 &=-\frac{8}{27}w_1^3+w_2^2+\frac{4}{3}w_1w_3
  -4w_4-\frac{32}{3}w_1w_6-48w_1^2w_8+48w_2w_7+288w_7^2\nn\\
 &\hspace{1em}
  -40w_3w_8-480w_6w_8-2592w_1w_8^2-9792w_8^3
  +\frac{1124}{3}w_1w_2+548w_5\nn\\
 &\hspace{1em}
  +6688w_1w_7+1884w_2w_8+25632w_7w_8+24576w_1^2
  +12920w_3+88320w_6\nn\\
 &\hspace{1em}
  +578688w_1w_8+1714176w_8^2-1694400w_2-8460000w_7-30102720w_1\nn\\
 &\hspace{1em}
  -104198400w_8+721612800.
\label{eq:abq0}
\end{align}
Here, $w_j\ (j=1,\ldots,8)$ denote
the Weyl orbit characters associated with 
the fundamental weights $\vecL_j^{E_8}$ of $E_8$
(see Appendix~\ref{app:E8} for our convention)
\begin{align}
w_j(\vecz)
 :=\sum_{\vecv\in\mbox{\scriptsize{Weyl orbit of }}\vecL_j^{E_8}}
   e^{2\pi \ri\vecv\cdot\vecz}.
\end{align}
The Weyl orbit of a weight $\bvec{\Lambda}$ is
the set of all weights obtained from $\bvec{\Lambda}$
by the action of the Weyl group.

\begin{prop}
\label{lem:ab0indep}
$\{a_i^{(0)}\}_{i=2}^4,\{b_j^{(0)}\}_{j=2}^6$
are algebraically independent over $\bbC$.
\end{prop}
\begin{proof}
This follows from the algebraic independence of
the Weyl orbit characters $\{w_j\}_{j=1}^8$.
Using the expression \eqref{eq:abq0}
one can compute the Jacobian determinant
\begin{align}
\left|
\frac{\partial(a_2^{(0)},a_3^{(0)},a_4^{(0)},
               b_2^{(0)},b_3^{(0)},b_4^{(0)},b_5^{(0)},b_6^{(0)})}
     {\partial(w_1,w_2,w_3,w_4,w_5,w_6,w_7,w_8)}\right|
=\frac{16384}{3}
\ne 0,
\end{align}
which proves the proposition.
\end{proof}
%

\section{Main theorems}\label{sec:proof}

In this section we will prove the main theorems of this paper.
\begin{thm}[{Conjecture of \cite[Sec.3.2]{Sakai:2017ihc}}]
\label{thm:subset}
The bigraded algebra $J^{E_8}_{*,*}$
of $W(E_8)$-invariant weak Jacobi forms
is a proper subset of the polynomial algebra
generated by $\{a_i\}_{i=2}^4, \{b_j\}_{j=1}^6$
over $M_*=\bbC[E_4,E_6]$. In other words,
\begin{align}
J^{E_8}_{*,*} \subsetneq
\bbC[E_4,E_6,a_2,a_3,a_4,b_1,b_2,b_3,b_4,b_5,b_6].
\end{align}
\end{thm}
\begin{proof}
Let $\cR$ denote the polynomial algebra
generated by $a_i,b_j$ over $M_*$:
\begin{align}
\cR:=\bbC[E_4,E_6,a_2,a_3,a_4,b_1,b_2,b_3,b_4,b_5,b_6].
\end{align}
By Theorem~\ref{thm:SunWang} (3) of Sun and Wang,
any $W(E_8)$-invariant weak Jacobi form
$\phi_t$ of index $t$ is expressed as in \eqref{eq:SunWangform}:
\begin{align}
\phi_t=
\frac{\sum_{j=0}^{t_1}\left({P_{16,5}}/{E_4}\right)^{t_1-j}P_j}
     {\Delta^{N_t}}.
\label{eq:SunWangformbis}
\end{align}
Proposition~\ref{prop:ABinab} states that
$A_i,B_j\in\cR$, so that $P_j\in\cR$.
Moreover, substituting \eqref{eq:ABinab} into \eqref{eq:P16c5}
one obtains
\begin{align}
\frac{P_{16,5}}{E_4}
&=\frac{1}{9216}(24\Delta E_4^2E_6a_2b_1b_2-18\Delta E_4^2E_6a_3b_1^2
 +20736\Delta E_4^2a_2b_1^3+5\Delta E_4E_6^2a_2^2b_1\nn\\
&\hspace{2em}
 -28440E_6^2b_1^5-336\Delta^2E_4E_6a_2a_3+4824\Delta E_6^2b_1^2b_3
 -1008\Delta E_6^2b_1b_2^2\nn\\
&\hspace{2em}
 -991872\Delta E_6b_1^3b_2-13436928\Delta b_1^5-384\Delta^2E_6^2b_5
 +27648\Delta^2E_6b_1b_4\nn\\
&\hspace{2em}
 +76032\Delta^2E_6b_2b_3-12690432\Delta^2b_1b_2^2),
\label{eq:P12c5inab}
\end{align}
which means that $P_{16,5}/E_4\in\cR$.
Hence, \eqref{eq:SunWangformbis} implies that
$\Delta^{N_t}\phi_t\in\cR$,
i.e.~it is
written as some polynomial $Q$ of $E_4,E_6,a_i,b_j$:
\begin{align}
\Delta^{N_t}\phi_t=Q(E_4,E_6,\{a_i\},\{b_j\}).
\end{align}

Since $\Delta=q+O(q^2)$ and any $W(E_8)$-invariant Jacobi form
has a regular power series expansion in $q$,
the $q$-expansion of $\Delta^{N_t}\phi_t$ starts
at the order of $q^n,\ n\ge N_t$.
Therefore, the $O(q^0)$ part of $Q(E_4,E_6,\{a_i\},\{b_j\})$
has to vanish, i.e.~there is an algebraic relation among
the $O(q^0)$ parts of $E_4,E_6,a_i,b_j$, denoted by
$E_4^{(0)}, E_6^{(0)}, a_i^{(0)}, b_j^{(0)}$.
We see that $E_4^{(0)}=1,\ E_6^{(0)}=1,\ b_1^{(0)}=-4$,
while $\{a_i^{(0)}\}_{i=2}^4,\{b_j^{(0)}\}_{j=2}^6$
are algebraically independent
by Proposition~\ref{lem:ab0indep}.
Since $b_1$ is of weight $0$ and index $1$ while
$E_4,E_6$ are of weight $4,6$ respectively and of index $0$,
the only possible algebraic relation
compatible with the bigrading is of the form
$\bigl(E_4^{(0)}\bigr)^3-\bigl(E_6^{(0)}\bigr)^2=0$
multiplied by some polynomial
of $E_4^{(0)}, E_6^{(0)}, a_i^{(0)}, b_j^{(0)}$.
This means that the polynomial $Q$ is
divisible by $\Delta=(E_4^3-E_6^2)/1728$,
i.e.~it is written as $Q=\Delta Q_1$ with $Q_1\in\cR$.

The Fourier expansion of $Q_1$ starts at the order of
$q^n,\ n\ge N_t-1$. One can repeat the same discussion as above
and show that $Q_1$ is written as
$Q_1=\Delta Q_2$ with $Q_2\in\cR$.
In this way, one can show that $Q$ is written as
$Q=\Delta Q_1=\Delta^2 Q_2=\cdots=\Delta^{N_t}Q_{N_t}$
with $Q_{N_t}\in\cR$, which gives $\phi_t$.
Thus we have proved that any $\phi_t\in J^{E_8}_{*,*}$
satisfies $\phi_t\in \cR$.
Since $a_i,b_j$ are meromorphic and are not
the elements of $J^{E_8}_{*,*}$, $\cR$ is bigger than $J^{E_8}_{*,*}$.
Hence $J^{E_8}_{*,*}\subsetneq\cR$.
\end{proof}
\begin{rem*}
A similar embedding of $J^{E_8}_{*,*}$ into a polynomial algebra
was considered in \cite[Theorem 5.1]{Wang:2020pzq},
where the polynomial algebra is generated by
$A_i/g^i,B_j/g^j$ with $g$ being the product of $\Delta^8E_4$
and a modular form of weight $72$.
Our Theorem~\ref{thm:subset} is stronger than this theorem.
\end{rem*}

The following theorem is rather a corollary of
Theorem~\ref{thm:SunWang} of Sun and Wang,
but is very useful when combined with Theorem~\ref{thm:subset}.
\begin{thm}\label{thm:E8cond}
Suppose that
$P\in\cR=\bbC[E_4,E_6,a_2,a_3,a_4,b_1,b_2,b_3,b_4,b_5,b_6]$.
Then the following two conditions are equivalent
\begin{align}
P\in J^{E_8}_{*,*}\quad\Leftrightarrow\quad
P\in\bbC[\Delta^{-1},E_4,E_6,A_1,A_2,A_3,A_4,A_5,B_2,B_3,B_4,B_6,
         (P_{16,5}/E_4)].
\end{align}
\end{thm}
\begin{proof}
$(\Rightarrow)$ This is a direct consequence of
Theorem~\ref{thm:SunWang} (3).
$(\Leftarrow)$
Any element of $\cR$ satisfies properties (i)--(iii) of
Definition~\ref{def:Jacobi} by construction
and also (iv) by Proposition~\ref{prop:abqexp}.
Therefore, we have only to show that $P$ is holomorphic
on $\bbH\times\bbC^8$.
The generators other than $(P_{16,5}/E_4)$ are holomorphic
by construction. $(P_{16,5}/E_4)$ is holomorphic
by Theorem~\ref{thm:SunWang} (1).
\end{proof}
%

\section{Constructing $W(E_8)$-invariant Jacobi forms of higher index}
\label{sec:results}

\subsection{Algorithm}

Using the theorems proved in the last section
we can formulate an efficient
algorithm for constructing all $W(E_8)$-invariant weak Jacobi forms
of given weight $k$ and index $m$, as described below:
\begin{algo}
\renewcommand{\theenumi}{\arabic{enumi}}
\renewcommand{\labelenumi}{(\theenumi)}
~\\[-3.5ex]
\begin{enumerate}
\item Let $P$ be the most general polynomial of $E_4,E_6,a_i,b_j$
of weight $k$ and index $m$. This is our ansatz.
More specifically, $P$ can be easily constructed by
taking the $O(x^ky^m)$ part of the generating series
\begin{align}
\frac{1}{(1-x^4E_4)(1-x^6E_6)
         \prod_{i=2}^4(1-x^{4-6i}y^ia_i)
         \prod_{j=1}^6(1-x^{6-6j}y^jb_j)}
\end{align}
and then inserting undetermined coefficient $c_i$
in front of every $i$th monomial.

\item Substitute \eqref{eq:abinABfull} into $P$
to express it in terms of $\Delta,E_4,E_6,A_i,B_j$.
In general, $P$ contains negative powers of $\Delta$ and $E_4$.
Let $-n$ be the lowest degree of $\Delta$.
Then multiply $P$ by $\Delta^n$, so that
$\Delta^n P$ no longer contains negative powers of $\Delta$.

\item Express $\Delta^n P$ in the form
\begin{align}
\Delta^n P
 =\sum_{l=1}^{l_1}
  \frac{Q_l(E_6,\{A_i\},\{B_j\})}{E_4^l}
 +R(E_4,E_6,\{A_i\},\{B_j\}),
\label{eq:PinQR}
\end{align}
where $l_1$ is some positive integer and
$Q_l,R$ are some polynomials.

\item For every $l=1,\ldots, l_1$,
let $S_l(E_6,\{A_i\},\{B_j\})$ be the most general polynomial
of weight $k+12n-12l$ and index $m-5l$
with undetermined coefficients $d_{l,i}$.
More specifically, $S_l$ can be easily constructed by
taking the $O(x^{k+12n-12l}y^{m-5l})$ part of the generating series
\begin{align}
\frac{1}{(1-x^6E_6)
         \prod_{i=1,2,3,4,5}(1-x^4 y^i A_i)
         \prod_{j=2,3,4,6}(1-x^6 y^j B_j)}
\end{align}
and then inserting $d_{l,i}$
in front of every $i$th monomial.

\item Solve the linear equations among $c_i$ and $d_{l,i}$
in such a way that
\begin{align}
Q_l(E_6,\{A_i\},\{B_j\})=(P_{16,5})^l S_l(E_6,\{A_i\},\{B_j\})
\quad (l=1,\ldots,l_1)
\label{eq:QScond}
\end{align}
hold identically.

\item Substitute the general solution back into $P$.
This gives the most general linear combination of
$W(E_8)$-invariant weak Jacobi forms of weight $k$ and index $m$.
If the trivial solution $c_i=d_{l,i}=0$ is the only solution,
there are no non-zero
$W(E_8)$-invariant weak Jacobi forms of weight $k$ and index $m$.
\end{enumerate}
\end{algo}

As an illustration, let us first construct
all $W(E_8)$-invariant weak Jacobi forms of weight $-16$ and index $5$.
The ansatz takes the form
\begin{align}
P=c_1 E_4^2 b_5
 +c_2 E_6 a_2 a_3
 +c_3 E_4 a_2 b_3
 +c_4 E_4 a_3 b_2
 +c_5 E_4 a_4 b_1
 +c_6 a_2^2 b_1.
\label{eq:m16c5ans}
\end{align}
By substituting \eqref{eq:abinABfull} into $P$,
one finds that the lowest degree of $\Delta$ is $-3$.
Therefore, expanding $\Delta^3P$ as in \eqref{eq:PinQR}
one obtains
\begin{align}
\begin{aligned}
Q_1
 &=(-\tfrac{10}{3}c_1+20c_2+\tfrac{5}{3}c_3-\tfrac{20}{9}c_4
    -\tfrac{10}{9}c_5)E_6A_1^3B_2\\
 &\hspace{1em}
  +(-\tfrac{14}{3}c_2+\tfrac{35}{54}c_4+\tfrac{7}{27}c_5)E_6^2A_1^2A_3
  +(6c_2+\tfrac{1}{12}c_6)E_6^2A_1A_2^2,\\
Q_2
 &=(-10c_2+\tfrac{5}{9}c_3+\tfrac{5}{6}c_4+\tfrac{1}{6}c_5
    -\tfrac{1}{6}c_6)E_6^2A_1^3A_2,\\
Q_3
 &=(\tfrac{1}{3}c_1+4c_2-\tfrac{5}{9}c_3-\tfrac{5}{9}c_4
   -\tfrac{1}{12}c_5+\tfrac{1}{12}c_6)E_6^2A_1^5.
\end{aligned}
\end{align}
In the present case, all $S_l$ are trivial,
i.e.~$S_1=S_2=S_3=0$ because there are no polynomials of $E_6,A_i,B_j$
of $(\mbox{weight},\mbox{index})=(8,0),(-4,-5),(-16,-10)$.
Thus, \eqref{eq:QScond} become simply $Q_l=0$ for $l=1,2,3$.
In order for these equations to hold identically,
$\{c_i\}_{i=1}^6$ have to satisfy five linear equations
(which are not entirely independent with each other).
The equations are solved as
\begin{align}
c_3=\frac{18}{5}c_1+\frac{36}{5}c_2,\quad
c_4=-\frac{24}{5}c_1-\frac{108}{5}c_2,\quad
c_5=12c_1+72c_2,\quad
c_6=-72c_2.
\end{align}
Substituting this
back into the original ansatz \eqref{eq:m16c5ans},
one obtains
\begin{align}
\begin{aligned}
&c_1\left(E_4^2b_5+\frac{18}{5}E_4a_2b_3-\frac{24}{5}E_4a_3b_2
          +12E_4a_4b_1\right)\\
{}+{}&c_2\left(E_6a_2a_3+\frac{36}{5}E_4a_2b_3-\frac{108}{5}E_4a_3b_2
 +72E_4a_4b_1-72a_2^2b_1\right).
\end{aligned}
\end{align}
This is the most general linear combination of
$W(E_8)$-invariant weak Jacobi forms
of weight~$-16$ and index~$5$.
Clearly, $\mathrm{dim} J^{E_8}_{-16,5}=2$.

Next, let us consider the case of weight~$-26$ and index~$7$
as another example. The ansatz takes the form
\begin{align}
\begin{aligned}
P&=
 c_1E_4^2a_3a_4
+c_2E_6a_2b_5
+c_3E_6a_3b_4
+c_4E_6a_4b_3\\
&\hspace{1em}
+c_5E_4a_2^2a_3
+c_6E_4b_1b_6
+c_7E_4b_2b_5
+c_8E_4b_3b_4\\
&\hspace{1em}
+c_9a_2b_1b_4
+c_{10}a_2b_2b_3
+c_{11}a_3b_1b_3
+c_{12}a_3b_2^2
+c_{13}a_4b_1b_2.
\end{aligned}
\label{eq:Pex2}
\end{align}
In this case, one finds that $n=5$, $l_1=6$ and
$S_l$ are constructed as
\begin{align}
S_1=d_{1,1}E_6^3A_2,\qquad S_2=S_3=S_4=S_5=S_6=0.
\end{align}
The linear equations among $\{c_i\}_{i=1}^{13}$ and $d_{1,1}$
are solved as
\begin{align}
\begin{aligned}
c_2&= 20736d_{1,1},&
c_3&= -\tfrac{41472}{5}d_{1,1},&
c_6&= 746496d_{1,1},&
c_7&= -\tfrac{746496}{5}d_{1,1},\\
c_8&= \tfrac{746496}{25}d_{1,1},&
c_9&= -\tfrac{4478976}{5}d_{1,1},&
c_{10}&= \tfrac{4478976}{25}d_{1,1},&
c_{11}&= \tfrac{4478976}{5}d_{1,1},\\
c_{12}&= -\tfrac{8957952}{25}d_{1,1},&
c_1&=c_4=c_5=c_{13}=0.\hspace{-4em}
\end{aligned}
\end{align}
Substituting this back into \eqref{eq:Pex2}
and setting $d_{1,1}=25/20736$, one obtains
\begin{align}
\begin{aligned}
&25E_6a_2b_5
-10E_6a_3b_4
+900E_4b_1b_6
-180E_4b_2b_5
+36E_4b_3b_4\\
&
-1080a_2b_1b_4
+216a_2b_2b_3
+1080a_3b_1b_3
-432a_3b_2^2.
\end{aligned}
\end{align}
This is the unique generator of the vector space $J^{E_8}_{-26,7}$.

\subsection{Free modules of given index}

It was proved in \cite[Theorem 4.1]{Wang:2018fil} that
the space $J^{E_8}_{*,m}=\bigoplus_{k\in\bbZ}J^{E_8}_{k,m}$
is a free module over $M_*$ and the rank $r(m)$ is given by
the generating series
\begin{align}
\frac{1}{(1-x)(1-x^2)^2(1-x^3)^2(1-x^4)^2(1-x^5)(1-x^6)}
 =\sum_{m=0}^\infty r(m)x^m.
\end{align}
To understand the structure of $J^{E_8}_{*,m}$ for given index $m$,
it is sufficient to determine $r(m)$ generators.
It was proved in \cite[Proposition 6.1]{Wang:2018fil} that
the weight of non-zero $W(E_8)$-invariant weak Jacobi forms
of index $m$ is not less than $-5m$.
It was also proved in \cite[Proposition 6.4]{Sun:2021ije} that 
for any $m\ge 2$, the free module $J^{E_8}_{*,m}$ is generated
by Jacobi forms of non-positive weight.
Therefore, it is sufficient to construct
Jacobi forms of weight $k$ with $-5m\le k\le 0$ for $m\ge 2$.

In \cite{Sun:2021ije}, Sun and Wang determined
all generators of $J^{E_8}_{*,m}$ for $1\le m\le 13$.
Using our algorithm we determine
all generators of $J^{E_8}_{*,m}$ for $1\le m\le 20$.

\begin{prop}
Let $d_{k,m}$ denote the number of generators of weight $k$ of
$J^{E_8}_{*,m}$.
For $1\le m\le 20$, the Laurent polynomials
\begin{align}
P^\rw_m:=\sum_{k\in\bbZ}d_{k,m}x^k
\end{align}
are determined as
\begin{align}
P^\rw_1&=
 x^4,\qquad
P^\rw_2=
 x^{-4}+x^{-2}+1,\qquad
P^\rw_3=
 x^{-8}+x^{-6}+x^{-4}+x^{-2}+1,\nn\\
P^\rw_4&=
 x^{-16}+x^{-14}+x^{-12}+x^{-10}+2x^{-8}+x^{-6}+x^{-4}+x^{-2}+1,\nn\\
P^\rw_5&=
 2x^{-16}+2x^{-14}+3x^{-12}+2x^{-10}+2x^{-8}+x^{-6}+x^{-4}+x^{-2}+1,
\nn\\
P^\rw_6&=
 2x^{-24}+2x^{-22}+3x^{-20}+3x^{-18}+3x^{-16}+3x^{-14}+3x^{-12}
 +2x^{-10}+2x^{-8}\nn\\
&\hspace{1em}
 +x^{-6}+x^{-4}+x^{-2}+1,\nn\\
P^\rw_7&=
 x^{-26}+3x^{-24}+5x^{-22}+7x^{-20}+4x^{-18}+4x^{-16}+4x^{-14}+3x^{-12}
 +2x^{-10}\nn\\
&\hspace{1em}
 +2x^{-8}+x^{-6}+x^{-4}+x^{-2}+1,\nn\\
P^\rw_8&=
 2x^{-32}+4x^{-30}+7x^{-28}+6x^{-26}+7x^{-24}+6x^{-22}+6x^{-20}+5x^{-18}
 +5x^{-16}\nn\\
&\hspace{1em}
 +4x^{-14}+3x^{-12}+2x^{-10}+2x^{-8}+x^{-6}+x^{-4}+x^{-2}+1,\nn\\
P^\rw_9&=
 x^{-36}+2x^{-34}+8x^{-32}+10x^{-30}+11x^{-28}+9x^{-26}+9x^{-24}
 +7x^{-22}+7x^{-20}\nn\\
&\hspace{1em}
 +6x^{-18}+5x^{-16}+4x^{-14}+3x^{-12}+2x^{-10}+2x^{-8}+x^{-6}+x^{-4}
 +x^{-2}+1,\nn\\
P^\rw_{10}&=
 4x^{-40}+7x^{-38}+11x^{-36}+12x^{-34}+14x^{-32}+12x^{-30}+12x^{-28}
 +11x^{-26}\nn\\
&\hspace{1em}
 +10x^{-24}+8x^{-22}+8x^{-20}+6x^{-18}+5x^{-16}+4x^{-14}+3x^{-12}
 +2x^{-10}+2x^{-8}\nn\\
&\hspace{1em}
 +x^{-6}+x^{-4}+x^{-2}+1,\nn\\
P^\rw_{11}&=
 5x^{-42}+15x^{-40}+19x^{-38}+20x^{-36}+16x^{-34}+17x^{-32}+15x^{-30}
 +14x^{-28}\nn\\
&\hspace{1em}
 +12x^{-26}+11x^{-24}+9x^{-22}+8x^{-20}+6x^{-18}+5x^{-16}+4x^{-14}
 +3x^{-12}+2x^{-10}\nn\\
&\hspace{1em}
 +2x^{-8}+x^{-6}+x^{-4}+x^{-2}+1,\nn\\
\noalign{\break}
P^\rw_{12}&=
 8x^{-48}+13x^{-46}+21x^{-44}+22x^{-42}+22x^{-40}+22x^{-38}+22x^{-36}
 +20x^{-34}\nn\\
&\hspace{1em}
 +20x^{-32}+17x^{-30}+15x^{-28}+13x^{-26}+12x^{-24}+9x^{-22}+8x^{-20}
 +6x^{-18}\nn\\
&\hspace{1em}
+5x^{-16}+4x^{-14}+3x^{-12}+2x^{-10}+2x^{-8}+x^{-6}+x^{-4}+x^{-2}+1,
\nn\\
P^\rw_{13}&=
 2x^{-52}+10x^{-50}+24x^{-48}+32x^{-46}+37x^{-44}+28x^{-42}+29x^{-40}
 +28x^{-38}\nn\\
&\hspace{1em}
 +26x^{-36}+23x^{-34}+22x^{-32}+18x^{-30}+16x^{-28}+14x^{-26}+12x^{-24}
 +9x^{-22}\nn\\
&\hspace{1em}
 +8x^{-20}+6x^{-18}+5x^{-16}+4x^{-14}+3x^{-12}+2x^{-10}+2x^{-8}+x^{-6}
 +x^{-4}+x^{-2}\nn\\
&\hspace{1em}
 +1,\nn\\
P^\rw_{14}&=
 9x^{-56}+22x^{-54}+37x^{-52}+38x^{-50}+39x^{-48}+37x^{-46}+38x^{-44}
 +36x^{-42}\nn\\
&\hspace{1em}
 +35x^{-40}+32x^{-38}+29x^{-36}+25x^{-34}+23x^{-32}+19x^{-30}+17x^{-28}
 +14x^{-26}\nn\\
&\hspace{1em}
 +12x^{-24}+9x^{-22}+8x^{-20}+6x^{-18}+5x^{-16}+4x^{-14}+3x^{-12}
 +2x^{-10}+2x^{-8}\nn\\
&\hspace{1em}
 +x^{-6}+x^{-4}+x^{-2}+1,\nn\\
P^\rw_{15}&=
 5x^{-60}+19x^{-58}+44x^{-56}+55x^{-54}+55x^{-52}+48x^{-50}+49x^{-48}
 +46x^{-46}\nn\\
&\hspace{1em}
 +46x^{-44}+42x^{-42}+39x^{-40}+35x^{-38}+31x^{-36}+26x^{-34}+24x^{-32}
 +20x^{-30}\nn\\
&\hspace{1em}
 +17x^{-28}+14x^{-26}+12x^{-24}+9x^{-22}+8x^{-20}+6x^{-18}+5x^{-16}
 +4x^{-14}\nn\\
&\hspace{1em}
 +3x^{-12}+2x^{-10}+2x^{-8}+x^{-6}+x^{-4}+x^{-2}+1,\nn\\
P^\rw_{16}&=
 16x^{-64}+37x^{-62}+58x^{-60}+63x^{-58}+65x^{-56}+60x^{-54}+62x^{-52}
 +61x^{-50}\nn\\
&\hspace{1em}
 +59x^{-48}+54x^{-46}+52x^{-44}+46x^{-42}+42x^{-40}+37x^{-38}+32x^{-36}
 +27x^{-34}\nn\\
&\hspace{1em}
 +25x^{-32}+20x^{-30}+17x^{-28}+14x^{-26}+12x^{-24}+9x^{-22}+8x^{-20}
 +6x^{-18}\nn\\
&\hspace{1em}
 +5x^{-16}+4x^{-14}+3x^{-12}+2x^{-10}+2x^{-8}+x^{-6}+x^{-4}+x^{-2}+1,
\nn\\
P^\rw_{17}&=
 6x^{-68}+32x^{-66}+73x^{-64}+89x^{-62}+90x^{-60}+76x^{-58}+79x^{-56}
 +76x^{-54}\nn\\
&\hspace{1em}
 +75x^{-52}+71x^{-50}+67x^{-48}+60x^{-46}+56x^{-44}+49x^{-42}+44x^{-40}
 +38x^{-38}\nn\\
&\hspace{1em}
 +33x^{-36}+28x^{-34}+25x^{-32}+20x^{-30}+17x^{-28}+14x^{-26}+12x^{-24}
 +9x^{-22}\nn\\
&\hspace{1em}
 +8x^{-20}+6x^{-18}+5x^{-16}+4x^{-14}+3x^{-12}+2x^{-10}+2x^{-8}+x^{-6}
 +x^{-4}+x^{-2}\nn\\
&\hspace{1em}
 +1,\nn\\
P^\rw_{18}&=
 26x^{-72}+59x^{-70}+95x^{-68}+102x^{-66}+98x^{-64}+96x^{-62}+99x^{-60}
 +96x^{-58}\nn\\
&\hspace{1em}
 +96x^{-56}+90x^{-54}+85x^{-52}+79x^{-50}+73x^{-48}+64x^{-46}+59x^{-44}
 +51x^{-42}\nn\\
&\hspace{1em}
 +45x^{-40}+39x^{-38}+34x^{-36}+28x^{-34}+25x^{-32}+20x^{-30}+17x^{-28}
 +14x^{-26}\nn\\
&\hspace{1em}
 +12x^{-24}+9x^{-22}+8x^{-20}+6x^{-18}+5x^{-16}+4x^{-14}+3x^{-12}
 +2x^{-10}+2x^{-8}\nn\\
&\hspace{1em}
 +x^{-6}+x^{-4}+x^{-2}+1,\nn\\
\noalign{\break}
P^\rw_{19}&=
 12x^{-76}+56x^{-74}+112x^{-72}+139x^{-70}+140x^{-68}+117x^{-66}
 +122x^{-64}\nn\\
&\hspace{1em}
 +122x^{-62}+119x^{-60}+113x^{-58}+110x^{-56}+100x^{-54}+93x^{-52}
 +85x^{-50}\nn\\
&\hspace{1em}
 +77x^{-48}+67x^{-46}+61x^{-44}+52x^{-42}+46x^{-40}+40x^{-38}+34x^{-36}
 +28x^{-34}\nn\\
&\hspace{1em}
 +25x^{-32}+20x^{-30}+17x^{-28}+14x^{-26}+12x^{-24}+9x^{-22}+8x^{-20}
 +6x^{-18}\nn\\
&\hspace{1em}
 +5x^{-16}+4x^{-14}+3x^{-12}+2x^{-10}+2x^{-8}+x^{-6}+x^{-4}+x^{-2}+1,
\nn\\
P^\rw_{20}&=
 34x^{-80}+93x^{-78}+151x^{-76}+159x^{-74}+152x^{-72}+145x^{-70}
 +151x^{-68}\nn\\
&\hspace{1em}
 +149x^{-66}+149x^{-64}+143x^{-62}+137x^{-60}+127x^{-58}+120x^{-56}
 +108x^{-54}\nn\\
&\hspace{1em}
 +99x^{-52}+89x^{-50}+80x^{-48}+69x^{-46}+62x^{-44}+53x^{-42}+47x^{-40}
 +40x^{-38}\nn\\
&\hspace{1em}
 +34x^{-36}+28x^{-34}+25x^{-32}+20x^{-30}+17x^{-28}+14x^{-26}+12x^{-24}
 +9x^{-22}\nn\\
&\hspace{1em}
 +8x^{-20}+6x^{-18}+5x^{-16}+4x^{-14}+3x^{-12}+2x^{-10}+2x^{-8}+x^{-6}
 +x^{-4}+x^{-2}\nn\\
&\hspace{1em}
+1.
\label{eq:Pwlists}
\end{align}
\end{prop}
For $1\le m\le 13$, our results of $P^\rw_m$
are in perfect agreement with those
of \cite[Theorem 4.2]{Sun:2021ije}.
As explained in \cite{Sun:2021ije},
$P^\rw_m$ gives the dimension of the space $J^{E_8}_{k,m}$
of weak Jacobi forms of arbitrary weight $k$ and given index $m$ as
\begin{align}
\frac{P^\rw_m}{(1-x^4)(1-x^6)}
 =\sum_{k\in\bbZ}\mathrm{dim} J^{E_8}_{k,m}\, x^m.
\end{align}
%

\subsection{Lowest weight generators of free modules}

We further construct all $W(E_8)$-invariant Jacobi forms
of weight $k\le -4m$ and index $m$ for $m\le 28$.
We find that there are
no $W(E_8)$-invariant Jacobi forms of weight $k<-4m$
and index $m$ for $1\le m\le 28$.
This serves as a further supporting evidence
of the following conjecture:
\begin{conj}[{Sun and Wang \cite[Conjecture 6.1]{Sun:2021ije}}]
\label{conj:lowerbound}
The weight of non-zero $W(E_8)$-invariant weak Jacobi forms
of index $m$ is not less than $-4m$.
\end{conj}

On the other hand, for every $12\le m\le 28$, we find 
$W(E_8)$-invariant Jacobi forms of weight $-4m$ and index $m$.
We summarize our results as follows:
\begin{prop}
Let the dimension of the space $J^{E_8}_{-4m,m}$ be described by
the series
\begin{align}
\mathcal{J}^{\mathrm{w},\lb}
 :=\sum_{m=0}^\infty\mathrm{dim}J^{E_8}_{-4m,m}x^m.
\end{align}
$\mathcal{J}^{\mathrm{w},\lb}$ is determined
up to the order of $x^{28}$ as
\begin{align}
\begin{aligned}
\mathcal{J}^{\mathrm{w},\lb}
&=1+x^4+2x^6+2x^8+x^9+4x^{10}+8x^{12}+2x^{13}+9x^{14}+5x^{15}
\\
&\hspace{1em}
+16x^{16}+6x^{17}+26x^{18}+12x^{19}+34x^{20}
+23x^{21}+52x^{22}
\\
&\hspace{1em}
+31x^{23}+80x^{24}+53x^{25}+105x^{26}+83x^{27}+154x^{28}
+O(x^{29}).
\end{aligned}
\end{align}
\end{prop}

Based on these results, we consider
the subspace of $J^{E_8}_{*,*}$ given by
\begin{align}
\Jlb
 :=\bigoplus_{m=0}^\infty J^{E_8}_{-4m,m}.
\end{align}
Clearly, $\Jlb$ is a graded subalgebra of $J^{E_8}_{*,*}$ over $M_*$.
Since Conjecture~\ref{conj:lowerbound} is verified for $1\le m\le 28$, 
at least within this range of $m$,
determining generators of weight $-4m$ and index $m$ of $J^{E_8}_{*,*}$
is equivalent to determining those of $\Jlb$.
We determine generators of index $m$ of $\Jlb$
for $1\le m\le 28$.
\begin{prop}
Let $\dlb_m$ denote the number of generators of index $m$ of $\Jlb$.
For $1\le m\le 28$, $\dlb_m$ are determined as in Table~\ref{table:dlb}.
\end{prop}
\begin{table}[t]
{\footnotesize
\begin{align*}
\begin{array}
{|@{\,}c@{\,}||
@{\,}c@{\,}|@{\,}c@{\,}|@{\,}c@{\,}|@{\,}c@{\,}|@{\,}c@{\,}|
@{\,}c@{\,}|@{\,}c@{\,}|@{\,}c@{\,}|@{\,}c@{\,}|@{\,}c@{\,}|
@{\,}c@{\,}|@{\,}c@{\,}|@{\,}c@{\,}|@{\,}c@{\,}|@{\,}c@{\,}|
@{\,}c@{\,}|@{\,}c@{\,}|@{\,}c@{\,}|@{\,}c@{\,}|@{\,}c@{\,}|
@{\,}c@{\,}|@{\,}c@{\,}|@{\,}c@{\,}|@{\,}c@{\,}|@{\,}c@{\,}|
@{\,}c@{\,}|@{\,}c@{\,}|@{\,}c@{\,}|}\hline
m&\bx{1}&\bx{2}&\bx{3}&\bx{4}&\bx{5}&\bx{6}&\bx{7}&
\bx{8}&\bx{9}&\bx{10}&\bx{11}&\bx{12}&\bx{13}&\bx{14}&
\bx{15}&\bx{16}&\bx{17}&\bx{18}&\bx{19}&\bx{20}&\bx{21}&
\bx{22}&\bx{23}&\bx{24}&\bx{25}&\bx{26}&\bx{27}&\bx{28}\\ \hline
\dlb_m&0&0&0&1&0&2&0&1&1&2&0&3&1&3&3&3&3&4&3&3&4&2&3&2&3&1&1&1\\ \hline
\end{array}
\end{align*}
\vspace{-3ex}
}
\caption{Number of generators of $\Jlb$}\label{table:dlb}
\end{table}
One sees that $\dlb_m>0$ for all $12\le m\le 28$.
This implies the following corollary:
\begin{cor}
The generators of $J^{E_8}_{*,*}$
must include those of weight $-4m$ and index $m$
with all $12\le m \le 28$.
\end{cor}
It is worth noting that $\Jlb$ is not freely generated,
i.e.~some elements of $\Jlb$ are not uniquely
expressed in terms of the generators.
This is clearly seen from
\begin{align}
\prod_{m=1}^\infty\frac{1}{(1-x^m)^{\dlb_m}}
 -\mathcal{J}^{\mathrm{w},\lb}
 =3x^{24}+2x^{25}+5x^{26}+6x^{27}+14x^{28}+O(x^{29}).
\end{align}
One sees that some algebraic relations among polynomials
of generators start appearing at index $24$.
This also means that $J^{E_8}_{*,*}$ is not freely generated.

The above fact and the existence of a large number of generators
give enough reason to conclude that
the structure of $J^{E_8}_{*,*}$ is highly complicated.
On the other hand, 
the modest behavior of $\dlb_m$ as a function of $m$
at large $m$ might be an indication that
$J^{E_8}_{*,*}$ is finitely generated,
as conjectured in \cite[Conjecture 6.7]{Sun:2021ije}.

\vspace{2ex}

\begin{center}
  {\bf Acknowledgments}
\end{center}

The author is grateful to Haowu Wang for discussions and
for valuable comments on a preliminary version of the manuscript.
The author would also like to thank Kaiwen Sun for discussions.
This work was supported in part by JSPS KAKENHI Grant Number 19K03856
and JSPS Japan--Russia Research Cooperative Program.

\vspace{2ex}


\appendix
\renewcommand{\theequation}{\Alph{section}.\arabic{equation}}

\section{Simple roots and fundamental weights of $E_8$}\label{app:E8}

Let $\{\vece_j\}\ (j=1,2,\ldots,8)$ be the orthonormal basis
of $\bbC^8$.
We take the simple roots of $E_8$ as
\begin{align}
\begin{aligned}
\vecal_1^{E_8}&=\tfrac{1}{2}\left(
  \vece_1-\vece_2-\vece_3-\vece_4
 -\vece_5-\vece_6-\vece_7+\vece_8
  \right),\\
\vecal_2^{E_8}&=\vece_1+\vece_2,\\
\vecal_j^{E_8}&=-\vece_{j-2}+\vece_{j-1}\quad(j=3,4,\ldots,8).
\end{aligned}
\end{align}
The fundamental weights of $E_8$ are then given by
\begin{align}
\begin{aligned}
\vecL_1^{E_8}&=2\vece_8,\\
\vecL_2^{E_8}&=\tfrac{1}{2}\vece_1+\tfrac{1}{2}\vece_2
  +\tfrac{1}{2}\vece_3+\tfrac{1}{2}\vece_4
  +\tfrac{1}{2}\vece_5+\tfrac{1}{2}\vece_6
  +\tfrac{1}{2}\vece_7+\tfrac{5}{2}\vece_8,\\
\vecL_3^{E_8}&=-\tfrac{1}{2}\vece_1+\tfrac{1}{2}\vece_2
  +\tfrac{1}{2}
\vece_3+\tfrac{1}{2}\vece_4+\tfrac{1}{2}\vece_5
  +\tfrac{1}{2}\vece_6+\tfrac{1}{2}\vece_7
  +\tfrac{7}{2}\vece_8,\\
\vecL_4^{E_8}&=\vece_3+\vece_4+\vece_5
  +\vece_6+\vece_7+5\vece_8,\\
\vecL_5^{E_8}&=\vece_4+\vece_5+\vece_6
  +\vece_7+4\vece_8,\\
\vecL_6^{E_8}&=\vece_5+\vece_6+\vece_7+3\vece_8,\\
\vecL_7^{E_8}&=\vece_6+\vece_7+2\vece_8,\\
\vecL_8^{E_8}&=\vece_7+\vece_8.
\end{aligned}
\end{align}
%

\section{Special functions}\label{app:functions}

The Jacobi theta functions are defined as
\begin{align}
\begin{aligned}
\varth_1(z,\tau)&:=
 \ri\sum_{n\in\bbZ}(-1)^n y^{n-1/2}q^{(n-1/2)^2/2},\\
\varth_2(z,\tau)&:=
  \sum_{n\in\bbZ}y^{n-1/2}q^{(n-1/2)^2/2},\\
\varth_3(z,\tau)&:=
  \sum_{n\in\bbZ}y^n q^{n^2/2},\\
\varth_4(z,\tau)&:=
  \sum_{n\in\bbZ}(-1)^n y^n q^{n^2/2},
\end{aligned}
\end{align}
where
\begin{align}
y=e^{2\pi \ri z},\qquad q=e^{2\pi \ri\tau}
\end{align}
and $z\in\bbC, \tau\in\bbH$.
We often use the following abbreviated notation
\begin{align}
\varth_k(\tau):=\varth_k(0,\tau).
\end{align}
The Dedekind eta function is defined as
\begin{align}
\eta(\tau):=q^{1/24}\prod_{n=1}^\infty(1-q^n).
\end{align}
The Eisenstein series are given by
\begin{align}
E_{2n}(\tau)
 =1-\frac{4n}{B_{2n}}\sum_{k=1}^{\infty}\frac{k^{2n-1}q^k}{1-q^k}
\end{align}
for $n\in\bbZ_{>0}$. The Bernoulli numbers $B_k$ are defined by
\begin{align}
\frac{x}{e^x-1}=\sum_{k=0}^\infty\frac{B_k}{k!}x^k.
\end{align}
We often abbreviate $\eta(\tau),\,E_{2n}(\tau)$
as $\eta,\,E_{2n}$ respectively.


\renewcommand{\section}{\subsection}
\renewcommand{\refname}

\end{document}